\documentclass[10pt,twoside]{article}
\def\YEAR{\year}\newcount\VOL\VOL=\YEAR\advance\VOL by-1995
\def\firstpage{1}\def\lastpage{1000}
\def\received{}\def\revised{}
\def\communicated{}

\makeatletter
\def\magnification{\afterassignment\m@g\count@}
\def\m@g{\mag=\count@\hsize6.5truein\vsize8.9truein\dimen\footins8truein}
\makeatother

\font\eightrm=cmr8
\font\caps=cmcsc10                    
\font\Caps=cmcsc10 scaled \magstep1   

\makeatletter
\@specialpagefalse\headheight=8.5pt
\def\DocMath{}
\renewcommand{\@evenhead}{%
    \ifnum\thepage>\lastpage\rlap{\thepage}\hfill%
    \else\rlap{\thepage}\slshape\leftmark\hfill{\caps\SAuthor}\hfill\fi}%
\renewcommand{\@oddhead}{%
    \ifnum\thepage=\firstpage{\DocMath\hfill\llap{\thepage}}%
    \else{\slshape\rightmark}\hfill{\caps\STitle}\hfill\llap{\thepage}\fi}%
\makeatother

\def\TSkip{\bigskip}
\newbox\TheTitle{\obeylines\gdef\GetTitle #1
\ShortTitle  #2
\SubTitle    #3
\Author      #4
\ShortAuthor #5
\EndTitle
{\setbox\TheTitle=\vbox{\baselineskip=20pt\let\par=\cr\obeylines%
\halign{\centerline{\Caps##}\cr\noalign{\medskip}\cr#1\cr}}%
	\copy\TheTitle\TSkip\TSkip%
\def\next{#2}\ifx\next\empty\gdef\STitle{#1}\else\gdef\STitle{#2}\fi%
\def\next{#3}\ifx\next\empty%
    \else\setbox\TheTitle=\vbox{\baselineskip=20pt\let\par=\cr\obeylines%
    \halign{\centerline{\caps##} #3\cr}}\copy\TheTitle\TSkip\TSkip\fi%
\centerline{\caps #4}\TSkip\TSkip%
\def\next{#5}\ifx\next\empty\gdef\SAuthor{#4}\else\gdef\SAuthor{#5}\fi%
\ifx\received\empty\relax
    \else\centerline{\eightrm Received: \received}\fi%
\ifx\revised\empty\TSkip%
    \else\centerline{\eightrm Revised: \revised}\TSkip\fi%
\ifx\communicated\empty\relax
    \else\centerline{\eightrm Communicated by \communicated}\fi\TSkip\TSkip%
\catcode'015=5}}\def\Title{\obeylines\GetTitle}
\def\Abstract{\begingroup\narrower
    \parskip=\medskipamount\parindent=0pt{\caps Abstract. }}
\def\EndAbstract{\par\endgroup\TSkip}

\long\def\MSC#1\EndMSC{\def\arg{#1}\ifx\arg\empty\relax\else
     {\par\narrower\noindent%
     2010 Mathematics Subject Classification: #1\par}\fi}

\long\def\KEY#1\EndKEY{\def\arg{#1}\ifx\arg\empty\relax\else
	{\par\narrower\noindent Keywords and Phrases: #1\par}\fi\TSkip}

\newbox\TheAdd\def\Addresses{\vfill\copy\TheAdd\vfill
    \ifodd\number\lastpage\vfill\eject\phantom{.}\vfill\eject\fi}
{\obeylines\gdef\GetAddress #1
\Address #2 
\Address #3
\Address #4
\EndAddress
{\def\xs{4.3truecm}\parindent=0pt
\setbox0=\vtop{{\obeylines\hsize=\xs#1\par}}\def\next{#2}
\ifx\next\empty 
     \setbox\TheAdd=\hbox to\hsize{\hfill\copy0\hfill}
\else\setbox1=\vtop{{\obeylines\hsize=\xs#2\par}}\def\next{#3}
\ifx\next\empty 
     \setbox\TheAdd=\hbox to\hsize{\hfill\copy0\hfill\copy1\hfill}
\else\setbox2=\vtop{{\obeylines\hsize=\xs#3\par}}\def\next{#4}
\ifx\next\empty\ 
     \setbox\TheAdd=\vtop{\hbox to\hsize{\hfill\copy0\hfill\copy1\hfill}
                \vskip20pt\hbox to\hsize{\hfill\copy2\hfill}}
\else\setbox3=\vtop{{\obeylines\hsize=\xs#4\par}}
     \setbox\TheAdd=\vtop{\hbox to\hsize{\hfill\copy0\hfill\copy1\hfill}
	        \vskip20pt\hbox to\hsize{\hfill\copy2\hfill\copy3\hfill}}
\fi\fi\fi\catcode'015=5}}\gdef\Address{\obeylines\GetAddress}

\usepackage{amssymb,latexsym,amsmath,amscd,amsthm,mathrsfs}
\usepackage{hyperref}
\usepackage{graphicx}

\newcommand{\eps}{\varepsilon}
\newcommand{\OO}{\mathscr{O}}
\newcommand{\FF}{\mathscr{F}}
\newcommand{\GG}{\mathscr{G}}
\newcommand{\HH}{\mathscr{H}}
\newcommand{\II}{\mathscr{I}}
\newcommand{\JJ}{\mathscr{J}}
\newcommand{\PP}{\mathbb{P}}
\newcommand{\ZZ}{\mathbb{Z}}
\newcommand{\Hom}{\mathrm{Hom}}
\newcommand{\Ext}{\mathrm{Ext}}
\newcommand{\cExt}{\mathcal{E}\!xt}
\newcommand{\cHom}{\mathscr{H}\!om}

\newtheorem{definition}{Definition}[section]
\newtheorem{theorem}[definition]{Theorem}
\newtheorem{proposition}[definition]{Proposition}
\newtheorem{remark}[definition]{Remark}
\newtheorem{corollary}[definition]{Corollary}
\newtheorem{code}[definition]{Code}

\begin{document}
\Title The Unirationality of Hurwitz Spaces 
of  $6$-Gonal Curves of Small Genus
\ShortTitle Hurwitz Schemes of  $6$-Gonal Curves 
\SubTitle   
\Author Florian Gei\ss
\ShortAuthor Florian Gei\ss 
\EndTitle
\Abstract 
	In this short note we prove the unirationality of Hurwitz spaces 
	of $6$-gonal curves of genus $g$ with $5\leq g\leq 28$ or $g=30,31,35,36,40,45$. Key ingredient is a liaison construction
	in $\PP^1 \times \PP^2$. By semicontinuity, the proof of the 
	dominance of this construction is reduced to a computation of a single curve
	over a finite field.
	\EndAbstract
\MSC 14H10, 
     14M20, 
     14M06, 
     14H51, 
     14Q05  
\EndMSC
\KEY 
     Unirationality, Hurwitz space.   
\EndKEY
\Address 
	Florian Gei\ss
	Mathematik und Informatik
	Universit\"at des Saarlandes, Campus E2 4, 
	66123 Saarbr\"ucken, 
	Germany
	fg@math.uni-sb.de
\Address
\Address
\Address
\EndAddress

\section{Introduction}
The study of the birational geometry of moduli spaces  
of curves with additional structures such as marked points 
or line bundles is a central topic in algebraic geometry, see 
for example the books \cite{HarrisMorrison} and \cite{ACGH2}.
The Hurwitz space $\HH(d,w)$ parametrizes $d$-sheeted 
branched simple covers of the projective line by smooth curves 
of genus $g$ with branch divisor of degree $w=2g+2d-2$ up to 
isomorphism,
\begin{equation*}
	\HH(d,2g+2d-2)=\{C\xrightarrow{d:1}\PP^1\ \text{ simply branched }|\ 
                   C\text{ smooth of genus } g \}/\sim.
\end{equation*}
It is a classical result by Arbarello and Cornalba \cite{ArbCorFootnotes}
based on a work of Segre \cite{Segre} that theses spaces are 
unirational for all $d\leq 5$ and all $g\geq d-1$ and in few cases 
for higher gonality, namely for $d=6$ and $5\leq g \leq 10$ or $g=12$
and for $d=7$ and $g=7$.\\
In this paper we present the following extension of this result to 
significantly higher genus for $6$-gonal curves. 
\begin{theorem}
	Over an algebraically closed field of characteristic zero, 
	the Hurwitz spaces 	$\HH(6,2g+10)$ of $6$-gonal curves of genus $g$ are unirational for 
    \begin{equation} \label{goodvalues}
    5\leq g\leq 28\text{ or }g=30,31,33,35,36,40,45.
    \end{equation}
\end{theorem}
Our proof is based on the observation that a general $6$-gonal curve in $\PP^1 \times \PP^2$ can be linked in two steps to the union of a rational curve and a collection of lines. It turns out that for small genera this process can be reversed by starting with a general rational curve and general lines. \\
To show that the described construction yields a parametrization of the Hurwitz space, we only need to run the construction for a single curve over a finite field. Semicontinuity then ensures that all assumptions we made actually hold for 
an open dense subset of $\HH(6,2g+10)$ in characteristic zero.
Since the construction works a priori only for finitely many genera we settle for a 
computer-aided verification using the computer 
algebra system \textit{Macaulay2} \cite{M2}.\\ 
An immediate consequence of our approach is that in the considered cases the general $6$-gonal curve has a plane model as expected from Brill-Noether theory.
\begin{corollary}
For $g$ among (\ref{goodvalues}) and $d=\left\lceil\frac{2}{3}g+2\right\rceil$ the Brill-Noether locus $W^2_d(C)$ of a general curve $C\in \HH(6,10+2g)$ has a smooth generically reduced component of expected dimension $\rho(g,2,d)$.
\end{corollary}

\subsection*{Acknowledgements}
This work is part of my PhD thesis. I would like to thank my advisor 
Frank-Olaf Schreyer for valuable discussions and continuous support.
I am grateful to the DFG for support through the priority program SPP 1489.
\section{Preliminaries}
Throughout this paper, we fix the following notation: Let 
$\PP=\PP^1\times\PP^2$ be the product of the projective line 
and the projective plane over a field $K$ with projections
$\pi_1: \PP \rightarrow \PP^1$ and $\pi_2:\PP \rightarrow \PP^2$. 
For $a,b \in \ZZ$ we write 
\begin{equation*}
	\OO_\PP(a,b)=\OO_{\PP^1}(a)\boxtimes \OO_{\PP^2}(b) 
	            =\pi^*_1\OO_{\PP^1}(a)  \otimes \pi^*_2\OO_{\PP^2}(b) 
\end{equation*}
and denote with
$R = \bigoplus_{i,j} H^0(\PP,\OO_\PP(i,j)) \cong K[x_0,x_1,y_0,y_1,y_2]$ 
the bihomogeneous coordinate ring of $\PP$. By a curve $C$ in $\PP$, we mean an equidimensional subscheme of codimension $2$ which is locally a complete 
intersection. We say that $C$ is \emph{(geometrically) linked} to a curve  $C' \subset \PP$ by a complete 
intersection $X\subset \PP$ if $C$ and $C'$ have no common components and $C\cup C'=X$. 
The Chow ring of $\PP$ is generated by classes $\alpha$ and $\beta$ which are the pullback of a point in $\PP^1$
and the pullback of a hyperplane in $\PP^2$, respectively. The bidegree $(d_1,d_2)$ of a curve $C$ is given by  $d_1=[C].\alpha$ and $d_2=[C].\beta$.

As in the classical setting of liaison of subschemes in $\PP^n$, we have the 
following 
\begin{proposition}[Exact sequence of liaison]
	Let $C$ be a curve of bidegree $(d_1,d_2)$ that is linked to $C'$ via 
	a complete 	intersection $X$ defined by forms of bidegree $(a_1,b_1)$ 
	and $(a_2,b_2)$. We set $a=a_1+a_2$ and $b=b_1+b_2$. 
	\begin{itemize}
		\item[(a)] There is a exact sequence
				$$
					0 \to \omega_C \to \omega_X \to \OO_{C'}(a-2,b-3)\to 0.
				$$
		\item[(b)] The curve $C'$ has bidegree 
				$(d_1',d_2')=(b_1b_2-d_1, a_1b_2+a_2b_1-d_2)$
				and arithmetic genus 
				$p_a(C')=p_a(X)-\left(d_1(a-2)+d_2(b-3)+(1-p_a(C))\right)$.
	\end{itemize}
\end{proposition}
\begin{proof} To prove the first part, consider the standard exact sequence
$$ 0 \to \II_{C/X} \to \OO_X \to \OO_C \to 0 $$
and apply $\Hom_{\OO_\PP}(-,\omega_\PP)$. From the long exact sequence, we get
$$ 0 \to \omega_{C} \to \omega_X \to \cExt^2(\II_{C/X},\omega_\PP) \to 0$$
but $\cExt^2(\II_{C/X},\omega_\PP)=\OO_{C'}(a-2,b-3)$ since $C$ and $C'$ are linked by $X$.
The formula for the genus follows immediately.  To compute the bidegree, note that $[C]+[C']=[X]=(b_1b_2)\beta^2+(a_1b_2+a_2b_1)\alpha \beta$ in the Chow ring of $\PP$. \end{proof}
Recall the following well-known fact about minimal resolutions of points in the plane.
\begin{proposition}\label{HilbertBurchPlanePoints}
Let $\Delta$ be a collection of $\delta$ general points in $\PP^2$ and let $k$ be maximal under the condition $\eps=\delta-{k+1 \choose 2}\geq 0$. Then the minimal free resolution of $\OO_\Delta$ is of the form
\begin{equation*} 
0 \to \GG \to \FF \to \OO_{\PP^2} \to \OO_{\Delta} \to 0
\end{equation*}
with $\FF=\OO(-k)^{k+1-\eps}$ and $\GG=\OO(-k-1)^{k-2\eps}\oplus \OO(-k-2)^{\eps}$ if 
$2\eps \leq k$ and $\FF= \OO(-k)^{k+1-\eps}\oplus \OO(-k-1)^{2\eps-k}$ and $\GG= \OO(-k-2)^{\eps}$ else. 
\end{proposition}
\begin{proof} \cite{Gae51}
\end{proof}
We also note the following simple but useful criterion for the irreducibility of plane curves. Recall that a variety over a field $K$ is called absolutely irreducible if it is irreducible as a variety over the algebraic closure $\overline{K}$.

\begin{proposition}
Let $C$ be a plane curve of degree $d$ with $\delta\leq \frac{d(d-3)}{2}$ ordinary double points 
and no other singularities. If the singular locus $\Delta$ of $C$ has a resolution as in \ref{HilbertBurchPlanePoints} then $C$ is absolutely irreducible.
\end{proposition}
\begin{proof}
Assume that $C$ decomposes into two curves $C_1$ and $C_2$ of degree $d_1$ and $d_2$ defined by
homogeneous polynomials $f_1$ and $f_2$. By assumption, $C_1$ and $C_2$ intersect transversely 
in $d_1\cdot d_2$ distinct points. 
First, we reduce to the case $d_1,d_2 \leq k$ where $k=\left\lceil (\sqrt{9+8\delta}-3)/2 \right\rceil$ is the minimal degree of generators of $I_\Delta$. Clearly, the case that one of the generators
has degree strictly larger than $k+1$ is not possible since $I_\Delta \subset (f_1,f_2)$ is generated in
degree $k$ and (possibly) $k+1$. The cases $d_1=k+1$, say, and $d_2\leq k+1$ can be
excluded by considering the number of minimal generators of $I_\Delta$ in degrees $k$ and $k+1$. \\
We are left with the case $d_1,d_2 \leq k$. Trivially, we can assume that $\delta-d_1d_2\geq 0$. A polynomial of the form $sf_1 + tf_2$ of degree $k$ lies in
$I_\Delta$ if it vanishes at the remaining $\delta-d_1d_2$ points. Hence,
\begin{eqnarray*}
h^0(\II_\Delta(k)) &\geq& {k-d_1+2 \choose 2} + {k-d_2+2 \choose 2} - \delta + d_1d_2\\
&=& 2 { k+2 \choose 2} + {d-1 \choose 2} - (dk+1) -\delta
\end{eqnarray*}
But this is strictly larger than ${k+2 \choose 2} - \delta$ since $d\geq k+3$.
\end{proof}
Recall from \cite{ACGH1} the following facts from Brill-Noether theory: For a fixed smooth curve $C$ of genus 
$g$, the Brill-Noether locus
\begin{equation}
W^r_d(C) = \{ L \in \mathrm{Pic}^d(C) \ |\ h^0(L) \geq r+1 \}
\end{equation}
is of dimension at least equal to the Brill-Noether number 
\begin{equation}
\rho(g,r,d)= g-(r+1)(g-d+r).
\end{equation}
The tangent space at a linear series $L \in W^r_d(C) \smallsetminus W^{r+1}_d(C)$ is the dual of the cokernel of the Petri-map
\begin{equation}
\mu_L: H^0(C,L) \otimes H^0(C,\omega_C \otimes L^{-1})\to H^0(C,\omega_C)
\end{equation} 
Hence, $W^r_d(C)$ is smooth of dimension $\rho$ at $L$ if and only if $\mu_L$ is injective.
\begin{proposition} \label{BrillNoetherplane}
Let $C$ be a smooth curve of genus $g\geq 3$ with $|D|$ a base point free $\mathfrak{g}^2_d$, $d=\left\lceil \frac{2g}{3} +2 \right\rceil$, such that the image of $C$ under the associated map is a plane curve with $\delta={d-1 \choose 2}-g$ ordinary double points and no other singularities. If the singular locus $\Delta$ has a resolution as in \ref{HilbertBurchPlanePoints} then $|D|$ is a smooth point in $W^2_d(C)$.
\end{proposition}
\begin{proof} By adjunction, the Petri map for $\OO(D)$ can be identified with
$$H^0(\PP^2,\OO(1)) \otimes H^0(\PP^2,\II_\Delta(d-4)) \to H^0(\PP^2, \II_\Delta(d-3)).$$
Under the given assumptions the minimal degree of generators of $I_\Delta$ is precisely $k=d-4$. As $2\eps\leq k$ we see from the minimal free resolution of $I_\Delta$ that the Petri map is injective since there are no linear relations among the generators of degree $k$ and $k+1$.
\end{proof}

\section{Liaison construcion}
For $g\geq 5$, let $f:C \rightarrow \PP^1$ be an element of $\HH(6,10+2g)$ and let 
$\OO(D_1) = f^*\OO_{\PP^1}(1)$ be the $6$-gonal bundle. We assume that $C$ has a line
bundle $\OO(D_2)$ such that $|D_2|$ is a complete base point free $\mathfrak{g}^2_d$
with $d=d(g)=\left\lceil \frac{2g}{3} +2 \right\rceil$ minimal under the condition that 
the Brill-Noether number $\rho(g,2,d)\geq 0$. 
Suppose further that the map
\begin{equation}
\varphi: C\xrightarrow{|D_1|,|D_2|} \PP H^0(\OO(D_1)) \times \PP H^0(\OO(D_2)) = \PP
\end{equation}
is an embedding. In particular, this is the case when we assume that the plane model has only ordinary double points and no other singularities and for any node $p$ the points in the preimage of $p$ under $C\to \PP^2$ are not identified under the map to $\PP^1$.

Hence, we can identify $C$ with its image under $\varphi$. Furthermore, we assume that the map 
$H^0(\OO_\PP(a,3)) \to H^0(\OO_C(a,3))$ is of maximal rank for all $a \geq 1$.\\
To simplify matters, assume $g\equiv 0\ (12)$ for the moment. By the maximal rank assumption, we have
\begin{equation}  \label{cubics}
a_{\mathrm{Cubic}}:=\min\{a \ |\ H^0(\II_C(a,3))\neq 0\} = \frac{g}{4}
\end{equation}
and $h^0(\II_C(a_{\mathrm{Cubic}},3))=3$. Let $X=V(f_1,f_2)$ be the complete intersection defined 
by two general sections $f_i \in H^0(\II_{C}(a_i,b_i))$ of bidegrees $(a_1,b_1)=(a_2,b_2)=(a_{\mathrm{Cubic}},3)$. The curve $C'$, obtained 
by liaison of $C$ by $X$, is smooth of bidegree $(3,\frac{5}{6}g-2)$ and genus $g'=\frac{g}{2}-3$ with $h^0(\II_{C'}(a_{\mathrm{Cubic}},3))\geq 2$.\\
The geometric situation is understood best when thinking of $C$ as a family of collections of plane points over $\PP^1$. We expect the general fiber of $C$ to be a collection of $6$  points in $\PP^2$ which are cut out by $4$ cubics. We expect a finite number $\ell$ of distinguished fibers where the points lie on a conic as this is a codimension $1$ condition on the points. Since the residual three points under liaison are collinear exactly in the distinguished fibers we can compute $\ell$ by examining the geometry of $C'$. The projection of $C'$ to $\PP^2$ yields a divisor $D_2'$ of degree $d'>g'+2$. Our claim is that $\ell=d'-(g'+2)$.
Indeed, the image of $C'$ under the associated map 
\begin{equation}
\psi:C'\to \PP^1 \times \PP H^0(C',\OO(D_2')) =\PP^1 \times \PP^{d'-g'}
\end{equation}
 lies on the graph of the projection $S \to \PP^1$ where $S$ is a $3$-dimensional  scroll of degree $d'-g'-2$ swept out by the $3$-gonal series $|D_1'|$, i.e. 
\begin{equation} \label{trigonal}
\psi(C')\subset \PP^1\times S=\bigcup_{D \in |D'_1|} \{D\} \times \overline{D}. \end{equation}
See \cite{Schreyer} for a proof of this fact. $C'$ is obtained from $\psi(C')$ by projection from a linear subspace 
$\PP^1\times V\subset \PP^1 \times \PP^{d'-g'}$ of codimension $3$. A general space $V$ intersects $S$ in precisely $d'-g'-2$ points 
lying in distinct fibers over $\PP^1$. Clearly, under the projection the points of $D \in |D_1'|$ 
are mapped to $3$ collinear points if and only if $V$ meets the corresponding fiber of $S$.\\
To keep things neat, we consider again the 
case $g\equiv 0\ (12)$ which implies $\ell=\frac{1}{3}g-1$. Suppose further that $\ell \equiv 1\ (3)$.
If we assume that the map $H^0(\OO_\PP(a,2)) \to H^0(\OO_{C'}(a,2))$ is of maximal rank for all $a\geq 1$ then 
\begin{equation} \label{conics}
a_{\mathrm{Conic}}=\min\{a \ |\ H^0(\II_{C'}(a,2))\neq 0 \}= \frac{g'+2\ell+1}{3}
\end{equation}
and $h^0(\II_{C'}(a_\mathrm{Conic},2))=2$. Let $X'=V(f_1',f_2')$ be defined by two general forms $f_i' \in H^0(\II(a_i',b_i'))$ of bidegrees $(a_1',b_1')=(a_2',b_2')=(a_\mathrm{Conic},2)$ and let $C''$ denote the curve that is linked to $C'$ via $X'$.
The general fiber of $C''$ consists of a single point. In a distinguished fiber the conics of the complete intersection are reducible and have the line spanned by the points of the fiber of $C'$ as a common factor. 
Hence, $C''$ is a rational curve together with $\ell$ lines. The rational curve has
degree 
\begin{equation}
d''=\frac{g'+2\ell-2}{3}=\frac{7}{18}g - \frac{7}{3}.
\end{equation}
Turning things around we see that the difficulty lies in reversing the first linkage step. Indeed, a simple counting argument shows that for any $g$, the union $C''$ of $\ell$ general lines in $\PP$ and the graph of a general rational normal curve of degree $d''$ we have $$\min\{ a \in \ZZ |\ H^0(\II_{C''}(a,2)) \neq 0 \} =  \left\lceil \frac{2d''+3\ell}{5}\right\rceil-1\leq a_{\mathrm{Conics}}.$$
Hence, we always obtain a trigonal curve $C'$ as desired. However, for general choices of $C''$ and $X'$ we expect that the map $H^0(\OO_\PP(a,3)) \to H^0(\OO_{C'}(a_{\mathrm{Cubic}},3))$ is of maximal rank. In the case $g\equiv 0 (12)$, this restriction yields $h^0(\II_{C'}(a_{\mathrm{Cubic}},3)) =-\frac{g}{4}+12$, hence $g<48$. 
Checking all congruency classes of $g$, we expect that $C'$ can be linked to a general curve $C$ exactly in the cases    \begin{equation} \label{goodgs} 5 \leq g\leq 28 \text{ or } g=30,31,33,35,36,40,45. \end{equation}

Table \ref{tablealldata} lists the appearing numbers for all values of  $g$ in (\ref{goodgs}).

\begin{table}
$$\begin{array}{r|r|r|r|r|r|r|r}
g & d & (a_1,b_1), (a_2,b_2) & g' & d' & (a_1',b_1'), (a_2',b_2') & \ell & d'' \\ \hline
5 & 6 & (2,3), (2,3) & 2 & 6 & (3,2) , (2,2) & 2 & 2 \\ 
6 & 6 & (2,3), (1,3) & 0 & 3 & (1,2) , (1,2) & 1 & 0 \\ 
7 & 7 & (2,3), (2,3) & 1 & 5 & (2,2) , (2,2) & 2 & 1 \\ 
8 & 8 & (3,3), (2,3) & 2 & 7 & (3,2) , (3,2) & 3 & 2 \\ 
9 & 8 & (2,3), (2,3) & 0 & 4 & (2,2) , (2,2) & 2 & 2 \\ 
10 & 9 & (3,3), (3,3) & 4 & 9 & (4,2) , (4,2) & 3 & 4 \\ 
11 & 10 & (3,3), (3,3) & 2 & 8 & (4,2) , (4,2) & 4 & 4 \\ 
12 & 10 & (3,3), (3,3) & 3 & 8 & (4,2) , (3,2) & 3 & 3 \\ 
13 & 11 & (4,3), (3,3) & 4 & 10 & (5,2) , (4,2) & 4 & 4 \\ 
14 & 12 & (4,3), (4,3) & 5 & 12 & (6,2) , (5,2) & 5 & 5 \\ 
15 & 12 & (4,3), (4,3) & 6 & 12 & (5,2) , (5,2) & 4 & 4 \\ 
16 & 13 & (4,3), (4,3) & 4 & 11 & (5,2) , (5,2) & 5 & 4 \\ 
17 & 14 & (5,3), (5,3) & 8 & 16 & (7,2) , (7,2) & 6 & 6 \\ 
18 & 14 & (5,3), (4,3) & 6 & 13 & (6,2) , (6,2) & 5 & 6 \\ 
19 & 15 & (5,3), (5,3) & 7 & 15 & (7,2) , (7,2) & 6 & 7 \\ 
20 & 16 & (6,3), (5,3) & 8 & 17 & (8,2) , (8,2) & 7 & 8 \\ 
21 & 16 & (5,3), (5,3) & 6 & 14 & (7,2) , (6,2) & 6 & 6 \\ 
22 & 17 & (6,3), (6,3) & 10 & 19 & (9,2) , (8,2) & 7 & 8 \\ 
23 & 18 & (6,3), (6,3) & 8 & 18 & (9,2) , (8,2) & 8 & 8 \\ 
24 & 18 & (6,3), (6,3) & 9 & 18 & (8,2) , (8,2) & 7 & 7 \\ 
25 & 19 & (7,3), (6,3) & 10 & 20 & (9,2) , (9,2) & 8 & 8 \\ 
26 & 20 & (7,3), (7,3) & 11 & 22 & (10,2) , (10,2) & 9 & 9 \\ 
27 & 20 & (7,3), (7,3) & 12 & 22 & (10,2) , (10,2) & 8 & 10 \\ 
28 & 21 & (7,3), (7,3) & 10 & 21 & (10,2) , (10,2) & 9 & 10 \\ 
30 & 22 & (8,3), (7,3) & 12 & 23 & (11,2) , (10,2) & 9 & 10 \\ 
31 & 23 & (8,3), (8,3) & 13 & 25 & (12,2) , (11,2) & 10 & 11 \\ 
33 & 24 & (8,3), (8,3) & 12 & 24 & (11,2) , (11,2) & 10 & 10 \\ 
35 & 26 & (9,3), (9,3) & 14 & 28 & (13,2) , (13,2) & 12 & 12 \\ 
36 & 26 & (9,3), (9 ,3) & 15 & 28 & (13,2) , (13,2) & 11 & 13 \\ 
40 & 29 & (10,3), (10,3) & 16 & 31 & (15,2) , (14,2) & 13 & 14 \\ 
45 & 32 & (11,3), (11,3) & 18 & 34 & (16,2) , (16,2) & 14 & 16 \\  \hline
\end{array}$$

\caption{\label{tablealldata} Numerical data for all cases of the construction}
\end{table}

Summarizing, we obtain for $g$ among (\ref{goodgs}) the following unirational construction for curves in $\HH(6,10+2g)$:
\begin{enumerate}
\item[1.] We start with a general rational curve of degree $d''$ in $\PP$ together with a collection of $\ell$ general lines. Call the union $C''$.
\item[2.] We choose two general forms  $f_i' \in H^0(\II_{C''}(a_i',b_i'))$, $i=1,2$, that define a complete intersection $X'$ and obtain a
trigonal curve $C'=\overline{X'\smallsetminus C''}$ of degree $d'$ and genus $g'$.
\item[3.] We choose two general forms $f_i \in H^0(\II_{C'}(a_i,b_i))$, $i=1,2$, that define a complete intersection $X$ and obtain a
$6$-gonal curve $C=\overline{X\smallsetminus C'}$.
\end{enumerate}
It remains to show that the construction actually yields a parametrization of the Hurwitz spaces.
\newpage
\section{Proof of the dominance}
\begin{theorem}\label{theoremUnirationalComponent}
For all $(g,d)$ as in Table \ref{tablealldata}, there is a unirational component $H_g$ of the Hilbert scheme $\mathrm{Hilb}_{(6,d),g}(\PP)$ of curves in $\PP$ of bidegree $(6,d)$ and genus $g$. The generic
point of $H_g$ corresponds to a smooth absolutely irreducible curve $C$ such that the map $H^0(\OO_\PP(a,3))\to H^0(\OO_C(a,3))$ is of maximal for all $a>1$.
\end{theorem}
\begin{proof} The crucial part is to prove the existence of a curve with the desired properties.
Code \ref{codeConstruction} implements the construction above for any given value of $g$ in (\ref{goodgs})
and establishes the existence of a smooth and absolutely irreducible curve $C_p$ of given genus and bidegree defined over a prime field $\mathbb{F}_p$. This computation can be regarded as the reduction of a computation over $\mathbb{Q}$ which 
yields some curve $C_0$. This curve is already defined over the rationals, since all construction steps invoke only Groebner basis computations. By semicontinuity, $C_0$ is also smooth, absolutely irreducible and of maximal rank.\\
Again, by semicontinuity, there is a Zariski open neighborhood $U\subset \mathrm{Hilb}_{(6,d),g}(\PP)$ 
of points corresponding to smooth absolutely irreducible curves that fulfill the maximal rank condition. Let $\mathbb{A}^N$ be the parameter-space for all the choices made in the construction, i.e. the space of coefficients of the polynomials defining $C''$ and the complete intersections $X$ and $X'$. The construction then translates to a rational map $\mathbb{A}^N\dasharrow U$ defined over $\mathbb{Q}$ and we set $H_g$ to be the closure of the image of this map.
\end{proof}
\begin{remark}
We want to point out two issues concerning the computational verification:
\begin{enumerate}
\item The restriction to finite fields in the \emph{Macaulay2} computation in the appendix is only due to limitations in computational power. For very small values of $g$, i.e. $g\leq 15$, it is still possible to compute examples over the rationals  if all coefficients are chosen among integers of small absolute value.
\item The reduction of $C_0$ modulo $p$ gives a curve $C_p$ with desired properties for $p$ in an open part of $\mathrm{Spec}(\ZZ)$. Hence, the main theorem is also true in almost all characteristics $p$. One way to extend it to all prime numbers would be to keep trace of all denominators in a computation over the rationals and check case by case the primes where a bad reduction happens. This is computationally also out of reach at the moment.
\end{enumerate}
\end{remark}
It remains to show that there exists a dominant rational map from $H_g$ to the Hurwitz-scheme. 
\begin{theorem}
For $g$ among (\ref{goodgs}) and $H_g$ as in Theorem \ref{theoremUnirationalComponent} there is a
dominant rational map 
$$ H_d \dasharrow \HH(6,10+2g).$$
\end{theorem} \begin{proof} 
Using Code \ref{codeConstruction} again, we check for any given value of $g$ in (\ref{goodgs}) the existence of a point in $H_g$ corresponding to a smooth absolutely irreducible curve $C\subset \PP$ such that the projection 
onto $\PP^1$ is simply branched and the bundle $L_2 = \varphi^*\OO_\PP(0,1)$ is a smooth point in the corresponding $W^2_d(C)$. By semicontinuity, the loci of curves with this property is open and dense in $H_g$. Hence, we have a rational map $ H_g \dasharrow \HH(6,10+2g)$. The locus of curves in $\HH(6,10+2g)$ having a smooth component of the Brill-Noether loci of expected dimension is also open and contains the image
of $[C]$ under this map. Since $\HH(6,10+2g)$ is irreducible this locus is dense. This proves the theorem.\end{proof}
We want to emphasize the last statement in the proof:
\begin{corollary} \label{planeModels} For $g$ among (\ref{goodgs}) and $d=\left\lceil \frac{2}{3}g+2 \right\rceil$ the Brill-Noether locus $W^2_{d}(C)$ of a general curve $C\in \HH(6,10+2g)$ has a smooth generically reduced component of expected dimension $\rho$.
\end{corollary}

\section{Computational Verification}
The following Code for \textit{Macaulay2} \cite{M2} realizes the unirational
construction of a $6$-gonal curve of genus $g$ as in (\ref{goodgs}) over a finite field $K=\mathbb{F}_p$ with random choices for all parameters. \\
In order to explain the single steps in the computation,
we also print the most relevant parts of the output for the example case $g=24$. 
\begin{code}\label{codeConstruction}

We start with the following initialization:
{\small
\begin{verbatim}
   i1 : Fp=ZZ/32009; -- a finite field
        S=Fp[x_0,x1,y_0..y_2,Degrees=>{2:{1,0},3:{0,1}}];
           -- Cox-ring of P^1 x P^2
        m=ideal basis({1,1},S);
           -- irrelevant ideal
        setRandomSeed("HurwitzSpaces");
           -- initialization of the random number generator
\end{verbatim}}
The following functions handle the numerics of the construction:
{\small
\begin{verbatim}   
   i2 : expHilbFuncIdealSheaf=(g,d,a)->
           max(0,(a_0+1)*(a_1+2)*(a_1+1)/2-(a_0*d_0+a_1*d_1+1-g))
           -- expected number of sections of the ideal sheaf
        
        linkedGenus=(g,d,F,G)->(
           pX:=binomial(F_0+G_0-1,1)*binomial(F_1+G_1-1,2)-
               (F_0-1)*binomial(F_1-1,2)-(G_0-1)*binomial(G_1-1,2);
               -- genus of the complete intersection
           pX-d_0*(F_0+G_0-2)-d_1*(F_1+F_1-3)-1+g)
           -- genus of the linked curve
           
        linkedDegree=(g,d,F,G)->{F_1*G_1-d_0,F_0*G_1+G_0*F_1-d_1}
           -- bidegree of the linked curve
\end{verbatim}}
The first step is to determine the degree $d''$ of the rational curve 
and the number of lines $\ell$.
We start by computing the bidegrees of the forms that define the complete intersection for the linkage to the trigonal curve:
{\small
\begin{verbatim}
   i3 : g=24;
        d={6,ceiling(-g/3+g+2)}; 
           -- choose the second degree Brill-Noether general
        a=for i from 0 do 
            if expHilbFuncIdealSheaf(g,d,{i,3})!=0 then break i;
           -- find the minimal value a s.t. H^0(IC(a,3)) nonzero
        if expHilbFuncIdealSheaf(g,d,{a,3})==1 then 
           fX={{a+1,3},{a,3}} else fX={{a,3},{a,3}}; 
           -- choose bidegrees of forms for the complete intersection
        (d,fX)
   
   o3 = ({6, 18}, {{6, 3}, {6, 3}})
\end{verbatim}}
The genus and degree of the trigonal curve and the number of lines:
{\small
\begin{verbatim}
   i4 : g'=linkedGenus(g,d,fX_0,fX_1);
        d'=linkedDegree(g,d,fX_0,fX_1);
        l=d'_1-g'-2;
        (g',d',l)
   
   o4 = (9,{3,18},7)
\end{verbatim}}
We compute the bidegrees for the complete intersection for the linkage to the rational curve
{\small
\begin{verbatim}
   i5 : b=for i from 0 do 
            if expHilbFuncIdealSheaf(g',d',{i,2})!=0 then break i;
        if expHilbFuncIdealSheaf(g',d',{b,2})==1 then
           fX'={{b+1,2},{b,2}} else fX'={{b,2},{b,2}};
        d''=linkedDegree(g'+2*l,d'+{0,l},fX'_0,fX'_1);
        dRat={{ceiling(d''_1/2),1},{floor(d''_1/2),1}};
        (fX',d'')
        
   o5 = ({{8, 2}, {8, 2}}, {1, 7})
\end{verbatim}}
The second step is the actual construction: First, we choose a rational curve and random lines and compute the saturated vanishing ideal $I_{C''}$ of their union:
{\small
\begin{verbatim}
   i6 : ICrat=saturate(ideal random(S^1,S^(-dRat)),m);
        ILines=apply(l,i->ideal random(S^1,S^{{-1,0},{0,-1}}));
        time IC''=saturate(intersect(ILines|{ICrat}),ideal(x_0*y_0));
           -- used 1.29537 seconds
\end{verbatim}}
Next, we choose random forms in $I_{C''}$ of degree $b$ (resp. of $b+1$) that define the complete intersection $X'$ and compute the saturated vanishing ideal $I_{C'}$ of the trigonal curve $C'$. 
{\small
\begin{verbatim}
   i7 : IX'=ideal(gens IC'' * random(source gens IC'',S^(-fX')));
   	    IC'=IX':ICrat;
        time scan(l,i->IC'=IC':ILines_i);
        time IC'sat=saturate(IC',ideal(x_0*y_0));
           -- used 2.06236 seconds
           -- used 23.7319 seconds
\end{verbatim}}
In the final step, we compute the vanishing ideal of the $6$-gonal curve $C$ by linking $C'$ with a complete intersection $X$ given by random forms 
in $I_{C'}$ of degree $a$ (resp. $a+1$). 
{\small
\begin{verbatim}
   i8 : IX=ideal(gens IC'sat * random(source gens IC'sat,S^(-fX)));
   	    time IC=IX:IC';
        time ICsat=saturate(IC,ideal(x_0*y_0));
           -- used 15.7815 seconds
           -- used 3.84807 seconds
\end{verbatim}}
We check that $C$ is of maximal rank in the degrees $(a,3)$ by looking at the minimal generators of the saturated vanishing ideal:
{\small
\begin{verbatim}
   i9 : tally degrees ideal mingens gb ICsat
   
   o9 = Tally{{0, 18} => 1}
              {1, 14} => 5
              {1, 15} => 4
              {2, 8} => 2
              {2, 9} => 8
              {3, 6} => 9
              {4, 4} => 2
              {4, 5} => 8
              {5, 4} => 7
              {6, 3} => 3
              {7, 3} => 1
\end{verbatim}}

In order to check irreducibility, we compute the plane model $\Gamma$ of $C$:
{\small
\begin{verbatim}
   i10 : Sel=Fp[x_0,x_1,y_0..y_2,MonomialOrder=>Eliminate 2];
            -- eliminination order
         R=Fp[y_0..y_2]; -- coordinate ring of P^2
         IGammaC=sub(ideal selectInSubring(1,gens gb sub(ICsat,Sel)),R);
            -- ideal of the plane model
\end{verbatim}}~\\
We check that $\Gamma$ is a curve of desired degree and genus and its singular locus $\Delta$ consists only of ordinary double points:~\\

{\small
\begin{verbatim}
   i11 : distinctPoints=(J)->(
            singJ:=minors(2,jacobian J)+J;
            codim singJ==3)

   i12 : IDelta=ideal jacobian IGammaC + IGammaC; -- singular locus
         distinctPoints(IDelta)

   o12 = true

   i13 : delta=degree IDelta;
         dGamma=degree IGammaC;
         gGamma=binomial(dGamma-1,2)-delta;
         (dGamma,gGamma)==(d_1,g)

   o13 = true
\end{verbatim}}~\\
We compute the free resolution of $I_{\Delta}$:~\\
{\small
\begin{verbatim}
   i14 : time IDelta=saturate IDelta;
         betti res IDelta
            -- used 55.063 seconds 
             
                0 1 2
   o14 = total: 1 8 7
             0: 1 . .
             1: . . .
             2: . . .
             3: . . .
             4: . . .
             5: . . .
             6: . . .
             7: . . .
             8: . . .
             9: . . .
            10: . . .
            11: . . .            
            12: . . .
            13: . 8 .
            14: . . 7
\end{verbatim}}~\\
This is the resolution as expected. Hence, $C$ is absolutely irreducible by Proposition \ref{HilbertBurchPlanePoints} and $\OO(D_2)$ is a smooth point
of the Brill-Noether loci by Proposition \ref{BrillNoetherplane}.\\
It remains to verify that $C$ is actually smooth and simply branched. We compute the vanishing ideal $I_B\subset K[x_0,x_1]$ of the locus $B$ in $\PP^1$ of points with non-reduced fiber.
{\small
\begin{verbatim}
   i15 : gensICsat=flatten entries mingens ICsat;
         Icubics=ideal select(gensICsat,f->(degree f)_1==3);
            -- select the cubic forms
         Jacobian=diff(matrix{{y_0}..{y_2}},gens Icubics);
            -- compute the jacobian w.r.t. to vars of P^2
         IGraphB=minors(2,Jacobian)+Icubics;
         time IGraphBsat=saturate(IGraphB,ideal(x_0*y_0));
            -- used 60.2963 seconds
\end{verbatim}}
We check that the fibers over $B$ are disjoint from the preimages of the double points of the plane model. This shows that $C$ is smooth:
{\small
\begin{verbatim}
   i16 : time ISing=saturate(sub(IDelta,S)+IGraphBsat,ideal(S_0*S_2));
         degree ISing==0

   o16 = true
\end{verbatim}}
Finally, we verify that $B$ is reduced of expected degree $2g+10$ and hence that $C$ is simply branched.
{\small
\begin{verbatim}         
  i17 : time IGraphBsat=saturate(IGraphB,ideal(x_0*y_0));
        gensIGraphBsat=flatten entries mingens IGraphBsat;
        IB=ideal select(gensIGraphBsat,f->(degree f)_1==0);
        degree radical IB==2*g+10 
        
  o17 = true 
\end{verbatim}}

\end{code}

This code is available in form of a \emph{Macaulay2}-file from \cite{Code} for download. It takes approximately 5 hours CPU-time on a 2.4 GHz processor to check all cases.


\Addresses
\end{document}